\documentclass[reqno]{amsart}

\usepackage{amsfonts}
\usepackage{amssymb}

\usepackage{amscd}
\usepackage{pictexwd,dcpic}
\usepackage{graphicx}
\usepackage{tikz}
\usepackage{fullpage}
\usepackage{hyperref}
\hypersetup{
    colorlinks=true,
    linkcolor=blue,
    filecolor=magenta,
    urlcolor=cyan,
    citecolor=cyan,}
\usepackage[nameinlink,capitalize]{cleveref}
\usepackage{enumitem}

\title{Tensor product surfaces and quadratic syzygies}

\author{Matthew Weaver}
\address{School of Mathematical and Statistical Sciences, Arizona State University, Wexler Hall, Tempe AZ 85281}
\email{matthew.j.weaver@asu.edu}

\date{}

\newtheorem{thm}{Theorem}[section]
\newtheorem*{thm-nonum}{Theorem}

\newtheorem{prop}[thm]{Proposition}
\newtheorem{lemma}[thm]{Lemma}
\newtheorem{cor}[thm]{Corollary}
\numberwithin{equation}{section}

\theoremstyle{definition}
\newtheorem{rem}[thm]{Remark}
\newtheorem{set}[thm]{Setting}

\newtheorem{defn}[thm]{Definition}

\newtheorem{ex}[thm]{Example}

\newtheorem{conj}[thm]{Conjecture}

\Crefname{thm}{Theorem}{Theorems}
\Crefname{ex}{Example}{Examples}

\def\K{\mathcal{K}}

\def\O{\mathcal{O}}

\def\P{\mathbb{P}}

\def\S{\mathcal{S}}

\def\Z{\mathcal{Z}}

\def\T{\underline{T}}
\def\f{\underline{f}}

\def\dim{\mathop{\rm dim}}

\def\im{\mathop{\rm im}}
\def\hgt{\mathop{\rm ht}}
\def\rk{\mathop{\rm rank}}

\def\Span{\mathop{\rm Span}}

\def\bideg{\mathop{\rm bideg}}
\def\syz{\mathop{\rm syz}}

\usepackage[all]{xy}

\usepackage{nicematrix}

\begin{document}

\begin{abstract}
For $U\subseteq H^0(\O_{\P^1\times \P^1}(a,b))$ a four-dimensional vector space, a basis $\{p_0,p_1,p_2,p_3\}$ of $U$ defines a rational map $\phi_U:\,\P^1\times \P^1 \dashrightarrow \P^3$. The tensor product surface associated to $U$ is the closed image $X_U$ of the map $\phi_U$. These surfaces arise within the field of geometric modelling, in which case it is particularly desirable to obtain the implicit equation of $X_U$. In this paper, we study $X_U$ via the syzygies of the associated bigraded ideal $I_U=(p_0,p_1,p_2,p_3)$ when $U$ is free of basepoints, i.e. $\phi_U$ is regular. Expanding upon work of Duarte and Schenck \cite{DS16} for such ideals with a linear syzygy, we address the case that $I_U$ has a quadratic syzygy.
\end{abstract}

\maketitle


\section{Introduction}

A classical problem within algebraic geometry is to find the implicit equations of the image or graph of a rational map between projective spaces. This so-called \textit{implicitization problem} has been studied to great length by algebraic geometers and commutative algebraists alike. More recently, the problem has gained interest within the geometric modeling community for its applications to computer-aided geometric design (CAGD), see e.g. \cite{CGZ00,SC95,SGD97,SSQK94}. In this context, knowledge of the implicit equations of a curve or surface allows for more efficient computation and geometric rendering. For instance, determining whether a point lies on a surface is trivial given an implicit form, but requires solving a possibly large polynomial system of equations given a parametric form.

In this context, two situations often considered are rational maps $\P^2 \dashrightarrow \P^3$ and $\P^1\times \P^1\dashrightarrow \P^3$. Surfaces defined as the image of the first map are \textit{triangular surfaces}, whereas surfaces defined as the image of the second are called \textit{tensor product surfaces}. Implicitization of these surfaces has been studied extensively using a variety of techniques, such as Gr\"obner bases, resultants, and syzygies. Whereas these first two methods are classical and straightforward to implement, they are typically computationally intensive and slow. In contrast, syzygy techniques, and methods borrowed from the study of \textit{Rees algebras}, are much more effective, see e.g. \cite{Botbol11,BDD09,BC05,BJ03}. One such tool in this realm is the \textit{approximation complex} $\Z$ introduced by Herzog, Simis, and Vasconcelos \cite{HSV82,HSV83}, which has proven particularly useful for implicitization purposes \cite{Botbol11, BC05, Chardin06}. We refer the reader to \cite{Cox03} for a wonderful overview of syzygy methods used for implicitization.

In this paper, we study tensor product surfaces and adopt the following setting. Let $R=k[s,t,u,v]$ be a polynomial ring over an algebraically closed field $k$, bigraded by setting $\bideg s,t =(1,0)$ and $\bideg u,v = (0,1)$. We note that the global sections $H^0(\O_{\P^1\times \P^1}(a,b))$ correspond to the bigraded components $R_{a,b}$ of $R$. Let $U\subseteq R_{a,b}$ denote a subspace with basis $\{p_0,p_1,p_2,p_3\}$, such that $p_0,p_1,p_2,p_3$ have no common zeros on $\P^1\times \P^1$, i.e. $U$ is \textit{basepoint free}. With this, consider the regular map 
$$\phi_U:\,\P^1\times \P^1 \longrightarrow \P^3$$
defined by $\{p_0,p_1,p_2,p_3\}$ and let $X_U = \phi_U(\P^1\times \P^1) \subseteq \P^3$. Writing $I_U=(p_0,p_1,p_2,p_3)$, the assumption that $U$ is free of basepoints is equivalent to $\sqrt{I_U} = (s,t)\cap (u,v)$.

From \cite{BDD09}, it is well known that the implicit equation of $X_U$ can be determined from the approximation complex $\Z$ on a generating set of $I_U$ (see \Cref{Prelim Section} for details). As this complex relates to the module of syzygies of $I_U$, knowledge of a free resolution of $I_U$ is particularly fruitful in this direction. In \cite{SSV14}, tensor product surfaces $X_U$ for $U\subset R_{2,1}$ are studied through a free resolution of $I_U$, of which the possible shapes are determined. There it is noted that the existence of a \textit{linear} syzygy, in bidegree $(0,1)$ or $(1,0)$, on $I_U$ yields certain constraints. This idea is further developed in \cite{DS16}, where this phenomenon is shown to extend beyond generation in bidegree $(2,1)$. Furthermore, in \cite{DS16} it is shown how the implicit equation of $X_U$ can be determined from the presence of a linear syzygy, without the need of a full resolution of $I_U$. In this direction, we show in the present article that a similar result can be achieved if $I_U$ has a \textit{quadratic} syzygy.

The key aspect in the arguments of \cite{DS16} is that a syzygy in bidegree $(0,1)$, or $(1,0)$, allows one to construct a specific generating set of $I_U$. This, in turn, allows for more detailed study of the remaining syzygies, with respect to this generation. In particular, one may produce a subset of syzygies which determines the bigraded strand $\Z_{2a-1,b-1}$ of the approximation complex, and hence the implicit equation of $X_U$, following \cite{Botbol11}.

In the case that $I_U$ has a syzygy in either bidegree $(0,2)$ or $(2,0)$, we show that a similar phenomenon occurs, within a handful of cases. The main innovation presented here is the construction of a particular subspace $V \subseteq U$, determined by the quadratic syzygy (see \Cref{Quad Syz Section} for details), with $\dim V$ serving as the invariant for each case. The main results of this article, \Cref{Dim 2 main result thm,Dim 3 main result thm}, are summarized as follows.

\begin{thm-nonum}
Assume that $U\subseteq H^0(\O_{\P^1\times \P^1}(a,b))$ is basepoint free, with $b\geq 3$. Let $I_U$ denote the ideal of $U$, and assume that $I_U$ has a minimal first syzygy $Q$ in bidegree $(0,2)$, and no linear syzygy. Write $V$ to denote the subspace of $U$ associated to the syzygy $Q$.
\begin{enumerate}
    \item[(i)] If $\dim V=2$, then $I_U$ has two syzygies $S_1,S_2$ of bidegree $(a,b-2)$ such that $\langle Q,S_1,S_2\rangle$ determines the first differential $d_1$ of $\Z_{2a-1,b-1}$.

    \item[(ii)] If $\dim V=3$, then $I_U$ has one syzygy $S_1$ of bidegree $(a,b-2)$ and two syzygies $S_2,S_3$ of bidegree $(a,b-1)$ such that $\langle Q,S_1,S_2,S_3\rangle$ determines the first differential $d_1$ of $\Z_{2a-1,b-1}$.
\end{enumerate}
Moreover, following \cite{Botbol11}, the determinant of a $2ab\times 2ab$ matrix representation of $d_1$ is a power of the implicit equation $F$ of $X_U$.
\end{thm-nonum}

By symmetry, a similar result holds if $I_U$ has a syzygy in bidegree $(2,0)$ and $a\geq 3$. The key aspect of the result above is the formulation of the subspace $V$, from a suitable basis of which the additional syzygies are constructed. Furthermore, the description of these syzygies is formulaic, and so this process may be easily implemented into a computer algebra system, such as \texttt{Macaulay2} \cite{Macaulay2}. In particular, the methods presented here yield more efficient computation of the implicit equation of $X_U$, in this setting.

To illustrate this, we consider the following, which will be our running example for the case that $\dim V=2$.

\begin{ex}\label{Intro example dim V=2}
Suppose that
$$U=\{s^2u^3+t^2u^2v,s^2uv^2+t^2v^3,s^2v^3,t^2u^3\}\subseteq H^0(\O_{\P^1\times \P^1}(2,3))$$
and let $I_U$ denote the ideal associated to $U$. A computation shows that $I_U$ has a first syzygy in degree $(0,2)$. Moreover, upon constructing the subspace $V\subseteq U$, further computations show that $\dim V=2$ and also that $I_U$ has seven minimal first syzygies in bidegrees
\[
\begin{array}{ccccccc}
(0,2), & (2,1), & (2,1), &(0,4), &(2,3), &(4,2), &(6,1).
\end{array}
\]
Following \Cref{Dim 2 main result thm}, the syzygies of bidegree $(0,2)$, $(2,1)$, and $(2,1)$ are the columns of the matrix 

\begin{equation}\label{intro example syzygies}
M=\begin{bmatrix}
      v^2&0&t^2u\\
      -u^2&s^2v&0\\
      0&-s^2u-t^2v&0\\
0&0&-s^2u-t^2v\end{bmatrix}    
\end{equation} 
and moreover, these syzygies determine the first differential of $\Z_{(2a-1,b-1)} = \Z_{3,2}$. Hence by \Cref{Botbol implicit eqn}, these syzygies yield the implicit equation of $X_U$.
\end{ex}

We continue this example in \Cref{dim 3 section}, verifying the construction of the two syzygies in (\ref{intro example syzygies}) of bidegree $(2,1)$ from the syzygy in bidegree $(0,2)$. Additionally, we show precisely how a matrix representation of the differential $d_1$ is obtained. Following \Cref{Botbol implicit eqn}, the determinant of the resulting $12\times 12$ matrix is then a power of the implicit equation of $X_U$.

We briefly describe how this paper is organized. In \Cref{Prelim Section}, we provide the preliminary material required for the scope of this article. We review the construction of the approximation complex of \cite{HSV82,HSV83} and recall the techniques of \cite{Botbol11}, using this complex for the purpose of multigraded implicitization. In \Cref{Quad Syz Section}, we introduce the main setting of this article, and show that a syzygy in bidegree $(0,2)$ implies constraints on the generation of $I_U$. The subspace $V\subseteq U$ is introduced, and we consider two cases based on its dimension. In \Cref{dim 2 section}, we address the first case where $\dim V=2$, and produce a pair of additional syzygies that determine the implicit equation of $X_U$. In \Cref{dim 3 section}, we consider the case when $\dim V=3$ and produce a set of three additional syzygies which determine the implicit equation of $X_U$ in this setting. We conclude the paper in \Cref{Questions section} with some further observations and questions related to the results presented here.


\section{Preliminaries}\label{Prelim Section}

Here we briefly describe the preliminary material necessary for this paper. We begin by introducing the \textit{approximation complex} $\Z$ \cite{HSV82,HSV83} associated to the ideal $I_U$. Moreover, we then recall the applications of this complex to the implicitization of tensor product surfaces, as outlined in \cite{Botbol11}. We refer the reader to \cite{Chardin06} for a nice overview of the techniques presented here.

\subsection{Approximation Complex}

We recall the construction of the \textit{approximation complex} $\Z$, introduced by Herzog, Simis, and Vasconcelos \cite{HSV82,HSV83}. This complex may be defined more generally, but for our purposes we consider the following setting. Let
$$I= (f_0,\ldots,f_n) \subseteq R=k[x_0,\ldots,x_d]$$
be a homogeneous ideal of $R$, and consider the Koszul complex $\K(\f)$ on the sequence $\f=f_0,\ldots,f_n$, with differentials $d_i^f$. Moreover, for new indeterminates $\T=T_0,\ldots,T_n$, let $S=k[T_0,\ldots,T_n]$ and consider the Koszul complex $\K(\T)$ with differential $d_i^S$. We construct a hybrid complex from this data.

\begin{defn}
Writing $Z_i = \ker d_i^f$ to denote the $i^{\text{th}}$ Koszul cycle, the approximation complex $\Z$ is the complex 
$$\Z\,:\quad \cdots \rightarrow \Z_{i+1}\overset{d_{i+1}}{\longrightarrow}\Z_{i}\overset{d_{i}}{\longrightarrow} \Z_{i-1}\overset{d_{i-1}}{\longrightarrow} \cdots $$
where $\Z_i = S\otimes_k Z_i$ and $d_i = d_i^S$.
\end{defn}

A direct computation shows that $d_i^f d_{i+1}^S+d_{i+1}^S d_i^f =0$. Hence for any $g\in Z_i$, we see that $d_i^f d_{i+1}^S (g) =-d_{i+1}^S d_i^f(g) =0$, and so $d_i(g) \in \Z_{i-1}$ and these maps are well defined. Moreover, the fact that $\Z$ is a complex follows as its differentials are inherited from $\K(\T)$. Much like the Koszul complex, the approximation complex $\Z$ depends only on the ideal $I$, and not the choice of generating set.

Whereas the higher homology of $\Z$ is often obscure, we note that the zeroth homology is familiar. Notice that $Z_0=R$ and $Z_1=\syz(I)$, the module of syzygies on $I$. It follows that the first differential $d_1$ of $\Z$ maps $d_1\,:S\otimes_k\syz(I) \rightarrow S\otimes_k R$ by
\begin{equation}\label{d_1 syz map}
(a_0,\ldots,a_n)\longmapsto a_0T_0+\cdots+a_nT_n.
\end{equation}
With this, we see that $H_0(\Z) \cong \S(I)$, where $\S(I)$ is the \textit{symmetric algebra} of $I$.

\begin{rem}
With the description in (\ref{d_1 syz map}), one may determine the image of $d_1$ from a free resolution, or even a free presentation, of $I$. Indeed, if $R^m\overset{\varphi}{\rightarrow} R^{n+1} \rightarrow I\rightarrow 0$ is such a presentation, with syzygy matrix $\varphi$, then $\im d_1 = (\ell_1,\ldots,\ell_m) \subset R[T_0,\ldots,T_n] \cong S\otimes_k R$, where $[\ell_1\ldots \ell_m] = [T_0\ldots T_n]\cdot \varphi$.
\end{rem}

With this, we note that it suffices to understand the syzygy matrix $\varphi$ to determine the first differential $d_1$ of $\Z$. We make use of this observation throughout this article.

\subsection{Multigraded Implicitization}

We now recall the applications of the approximation complex $\Z$ to the implicitization of tensor product surfaces. This complex has been used in multiple instances for implicitization purposes \cite{BDD09,BC05,Chardin06}, but for conciseness, we refer to the tools developed in \cite{Botbol11}.

Notice that $R\otimes_k S$ is naturally bigraded, hence we
may take graded strands of $\Z$. Indeed, for $\nu$ a fixed degree within the grading of $R$, the complex $\Z_\nu$ is
$$\Z_\nu\,: \quad \cdots \longrightarrow S\otimes_k (Z_i)_\nu \overset{d_i}{\longrightarrow} S\otimes_k (Z_{i-1})_\nu \overset{d_{i-1}}{\longrightarrow} \cdots.$$
Moreover, we note that if $R$ is bigraded, as in the proceeding sections, one may also take a \textit{bigraded} strand $\Z_\nu$ consisting of $S$-modules as well.

\begin{lemma}[{\cite[7.3]{Botbol11}}]\label{Botbol det Z}
Let $U=\{p_0,p_1,p_2,p_3\} \subset R_{a,b}$ and let $\phi_U:\,\P^1\times\P^1 \dashrightarrow \P^3$ denote the rational map defined by $U$. Assume that either $\phi_U$ has no basepoints or $\phi_U$ has finitely many basepoints that are locally a complete intersection. Let $\nu=(2a-1,b-1)$ (equivalently $\nu =(a-1,2b-1)$) and let $\Delta_\nu= \det \Z_\nu$. Then
$$\deg (\Delta_\nu) = 2ab - \dim (H_2)_{4a-1,3b-1}.$$
Moreover, the differential $d_1:\, (\Z_1)_\nu \rightarrow (\Z_0)_\nu$ has a matrix representation which is square of size $2ab\times 2ab$ if and only if $(H_2)_{4a-1,3b-1}=0$.
\end{lemma}

Here $(H_2)_{4a-1,3b-1}$ denotes the second homology module of the bigraded strand $\Z_{4a-1,3b-1}$. The notion of the determinant of a complex is a general formulation, but we refer the reader to \cite{Chardin06} for a concise description aimed towards applications similar to those presented here. We will be most interested in the case that $U$ is free of basepoints, in which case $(H_2)_{4a-1,3b-1}$ vanishes and the determinant of $\Z_\nu$ is simply the determinant of a square matrix representation of $d_1$.

\begin{lemma}[{\cite[7.4]{Botbol11}}]\label{Botbol degree lemma}
With the conditions of \Cref{Botbol det Z}, suppose that the basepoints of $U$ (if any) have multiplicity $e_x$. One has
$$\deg (\phi_U)\deg(F) = 2ab -\sum e_x$$
where $F\in S=k[T_0,T_1,T_2,T_3]$ is the implicit equation of $X_U$.
\end{lemma}

Combining \Cref{Botbol det Z,Botbol degree lemma}, we obtain our primary tool to determine the implicit equation $F$ of $X_U$.

\begin{thm}[{\cite[7.5]{Botbol11}}]\label{Botbol implicit eqn}
    With the assumptions of \Cref{Botbol det Z}, we have that $\Delta_\nu = F^{\deg \phi_U}$. In particular, from  \Cref{Botbol degree lemma} we have
    $$\deg \Delta_\nu =\deg (F^{\deg \phi_U}) = (\deg F)(\deg \phi_U) =  2ab -\sum e_x.$$
    Hence by \Cref{Botbol det Z}, it follows that $\dim (H_2)_{4a-1,3b-1} =\sum e_x$.
\end{thm}

Thus in the absence of basepoints, the differential $d_1:\, (\Z_1)_\nu \rightarrow (\Z_0)_\nu$ is a square $2ab\times 2ab$ matrix. Additionally, its determinant is a power of the implicit equation $F$ of $X_U$, and this power is precisely $\deg \phi_U$.


\section{Quadratic Syzygies}\label{Quad Syz Section}

We now introduce the setting for the duration of the paper. Whereas many of the conventions have been stated in the introduction, we briefly restate them here for clarity. Our primary setting is the following.

\begin{set}\label{General Setting}
Let $R=k[s,t,u,v]$ with $\bideg s,t =(1,0)$ and $\bideg u,v = (0,1)$. Let $U \subseteq R_{a,b}$ be a subspace with basis $\{p_0,p_1,p_2,p_3\}$ and let $I_U = (p_0,p_1,p_2,p_3) \subseteq R$. Assume that $U$ is basepoint free and write $\phi_U:\, \P^1\times \P^1 \longrightarrow \P^3$ to denote the regular map defined by $U$, with image $X_U$. Assume that $b\geq 3$ and that $I_U$ has a first syzygy $Q$ of bidegree $(0,2)$. Moreover, assume that $I_U$ has no linear first syzygy.
\end{set}

We may safely assume that $I_U$ has no linear syzygy in bidegree $(0,1)$ or $(1,0)$, as this case has already been examined in \cite{DS16}. Moreover, we note that by symmetry, the case that $I_U$ has a syzygy in bidegree $(2,0)$ is also addressed if $a\geq 3$.

To begin our initial treatment in the setting above, we apply a technique from the proof of \cite[4.8]{SSV14}, and introduce similar constructions. As $I_U$ has a first syzygy of bidegree $(0,2)$, there exist coefficients $a_i,b_i,c_i \in k$ such that  \begin{equation}\label{syz eqn}
\sum_{i=0}^3(a_iu^2+b_iuv+c_iv^2) p_i=0.
\end{equation}
Rearranging, we have
\begin{equation}\label{syz sum}
0=\sum_{i=0}^3(a_iu^2+b_iuv+c_iv^2) p_i= (\sum_{i=0}^3a_i p_i)u^2 + (\sum_{i=0}^3 b_i p_i)uv + ( \sum_{i=0}^3 c_i p_i) v^2.   
\end{equation}
Writing $f_0 = \sum_{i=0}^3a_i p_i$, $f_1=\sum_{i=0}^3b_i p_i$, and $f_2 = \sum_{i=0}^3c_i p_i$, we note that $[f_0, f_1, f_2]$ is a syzygy on $[u^2, uv, v^2]$. As a free resolution of the ideal $(u^2, uv,v^2)$ is well known, by say the Hilbert-Burch theorem \cite[20.15]{Eisenbud}, it follows that 
\begin{equation}\label{alpha-beta columns}
\begin{bmatrix}
    f_0\\
    f_1\\
    f_2
\end{bmatrix} = \alpha \begin{bmatrix}
    v\\
    -u\\
    0
\end{bmatrix} + \beta \begin{bmatrix}
    0\\
    v\\
    -u
\end{bmatrix}
=\begin{bmatrix}
    \alpha v\\
    \beta v - \alpha u\\
    -\beta u
\end{bmatrix}
\end{equation}
for some $\alpha, \beta \in R_{a,b-1}$.

\begin{rem}\label{f0 and f2 nonzero}
Notice that, as $I_U$ has no linear syzygy, both $f_0 \neq 0$ and $f_2\neq 0$. Indeed, as $\{p_0,p_1,p_2,p_3\}$ is linearly independent, if $f_0=0$ we see that $a_0=a_1=a_2=a_3=0$. Hence from (\ref{syz eqn}) we have
$$0=\sum_{i=0}^3(b_iuv+c_iv^2) p_i = v\sum_{i=0}^3(b_iu+c_iv) p_i,$$
and so $\sum_{i=0}^3(b_iu+c_iv) p_i =0$, which contradicts the assumption that $I_U$ has no linear syzygy. A similar argument shows that $f_2 \neq 0$. In particular, following (\ref{alpha-beta columns}), we see that both $\alpha$ and $\beta$ are nonzero.
\end{rem}

We note that, whereas both $f_0$ and $f_2$ are nonzero, there is no such restriction on the vanishing of $f_1$. In the case that $f_1=0$, (\ref{syz sum}) and \Cref{Dim 2 generation} will show that $I_U$ has a \textit{reduced} Koszul syzygy in bidegree $(0,2)$, which may be taken as $Q=[v^2,-u^2,0,0]$; see \Cref{Intro example dim V=2}.

With this, we will eventually consider two cases: when the set $\{f_0,f_1,f_2\}$ is $k$-linearly independent and when it is not. To this end, let $V = \Span_k \{f_0,f_1,f_2\}$ denote the subspace of $U$ spanned by $\{f_0,f_1,f_2\}$.

\begin{prop}\label{Dimension of V}
With $V$ the subspace spanned by $\{f_0,f_1,f_2\}$, we have that $2\leq \dim V\leq 3$.
\end{prop}

\begin{proof}
The last inequality is clear, hence we need only verify the first. By \Cref{f0 and f2 nonzero}, we see that $V\neq 0$ and so it suffices to show that $\dim V\neq 1$. Recall from \Cref{f0 and f2 nonzero} that $f_0\neq 0$, hence if $\dim V=1$ then $f_1=d_1 f_0$ and $f_2=d_2 f_0$ for some $d_1,d_2\in k$. Thus (\ref{syz sum}) shows that $(u^2 + d_1uv +d_2v^2)f_0 =0$, however this is a contradiction as $R$ is a domain.
\end{proof}

\begin{rem}\label{lin syz case}
We note that the subspace $V$ is inspired by a similar construction in \cite[4.8]{SSV14}, within the study of tensor product surfaces of bidegree $(2,1)$. However, the formulation of such a subspace associated to a syzygy holds quite generally, a topic which we explore in \Cref{Questions section}. In particular, adapting this method to the setting of a \textit{linear} syzygy in \cite{DS16}, the proof of \cite[2.1]{DS16} shows the resulting subspace $V$ has $\dim V=2$ in this setting. As a consequence, the statement of \Cref{Dim 2 main result thm} reads very similarly to \cite[2.2]{DS16}.
\end{rem}

With \Cref{Dimension of V}, we may consider two cases, namely when $\dim V= 2$ and when $\dim V=3$. We note that this figure dictates the largest size of a subset of $\{f_0,f_1,f_2\}$ that may be taken as part of a minimal generating set of $I_U$. With this, we proceed as in \cite{DS16} and establish particular generating sets of $I_U$, in each case. Additional syzygies, based upon this generation, are then determined in the proceeding sections.

\begin{rem}\label{Generation - dim V remark}
We note that the dimension of $V$ is easily computed from the data of (\ref{syz sum}). Indeed, write $\varphi$ for the $4\times 3$ coefficient matrix
\begin{equation}\label{coefficient matrix}
    \varphi = \begin{bmatrix}
        a_0&b_0 &c_0 \\
        a_1&b_1&c_1\\
        a_2& b_2& c_2\\
        a_3&b_3&c_3
    \end{bmatrix}
\end{equation}
and note that $[f_0,f_1,f_2] = [p_0,p_1,p_2,p_3]\cdot \varphi$. As $\{p_0,p_1,p_2,p_3\}$ is a basis of $U$, and hence linearly independent, it follows that $\dim V = \rk \varphi$, the latter of which is a linear algebra computation.

Additionally, we note that $\dim V$ is precisely the largest size of a subset of $\{f_0,f_1,f_2\}$ that may be taken as part of a minimal generating set of $I_U$. Clearly one has $(f_0,f_1,f_2)\subseteq I_U = (p_0,p_1,p_2,p_3)$ and, as every polynomial involved has the same bidegree, the size of this subset is exactly the largest-sized nonzero, and hence invertible, minor of $\varphi$. 
\end{rem}

\subsection{Case 1: \boldmath{$\dim V=2$}}

With the conditions of \Cref{General Setting}, we consider the first case of \Cref{Dimension of V} and assume that $\dim V=2$. With this, we aim to produce a particular generating set of $I_U$, which will be used to determine additional syzygies of $I_U$ in \Cref{dim 2 section}. 

With the assumption that $\dim V=2$, clearly the set $\{f_0,f_1,f_2\}$ is linearly dependent. Hence there is an equation of dependence
\begin{equation}\label{lin dep eqn}
    d_0 f_0+d_1f_1+d_2f_2 =0
\end{equation}
for $d_0,d_1,d_2\in k$, with at least one nonzero. With this, we consider the (non-exclusive) subcases depending on the non-vanishing of the coefficients.

\begin{itemize}
\item[(i)] Suppose that $d_2\neq 0$. After rescaling $d_2$ to $1$, rearranging, and relabelling in (\ref{lin dep eqn}), we may write 
\begin{equation}\label{f2 equation}
f_2= d_0f_0+d_1f_1
\end{equation}
for $d_0,d_1\in k$. With this, we see that $\{f_0,f_1\}$ is then a basis for $V$. We also note that, as $f_2\neq 0$ by \Cref{f0 and f2 nonzero}, it follows that at least one of $d_0,d_1$ is nonzero. 

Recall from (\ref{syz sum}) that $[f_0,f_1,f_2]$ is a syzygy on the ideal $(u^2,uv,v^2)$. With this and (\ref{f2 equation}) we have
$$f_0u^2+f_1uv+(d_0f_0+d_1f_1)v^2=0,$$
hence
\begin{equation}\label{f0 f1 Koszul relation}
f_0(u^2+d_0v^2) + f_1(uv+d_1v^2)=0.    
\end{equation}
With this, we observe that $u^2+d_0v^2$, $uv+d_1v^2$ is an $R$-regular sequence. Indeed, clearly the two are not unit multiples of each other, and so $\hgt (u^2+d_0v^2, uv+d_1v^2) =1$ only if these polynomials have a common linear factor in $k[u,v]$. However, factoring it from (\ref{f0 f1 Koszul relation}), it would then follow that $I_U$ has a linear syzygy, which contradicts the assumptions of \Cref{General Setting}.

From (\ref{f0 f1 Koszul relation}) it follows that $[f_0, f_1]$ is a syzygy on the ideal $(u^2+d_0v^2, uv+d_1v^2)$. Since this ideal is generated by a regular sequence, its syzygy module is spanned by a single Koszul syzygy. Hence
\[
 \left\{
    \begin{array}{ll}
      f_0  =& h(uv+d_1v^2) \\
      f_1 =&  -h (u^2+d_0v^2)
    \end{array}
  \right.
  \]
  for some $h \in R_{a,b-2}$. Lastly, we note that since $u^2+d_0v^2$ and $uv+d_1v^2$ share no common factor, and $uv+d_1v^2 = v(u+d_1v)$, we have that
  $u+d_1v \nmid u^2+d_0v^2$. With this, we see that $d_1^2+d_0 \neq 0$, recalling that $d_0$ and $d_1$ cannot simultaneously vanish by (\ref{f2 equation}) and \Cref{f0 and f2 nonzero}.

\item[(ii)] Alternatively, suppose that $d_1\neq 0$. After rescaling $d_1$ to $1$, rearranging, and relabelling in (\ref{lin dep eqn}) accordingly, we may write  
\begin{equation}\label{f1 equation}
f_1= d_0f_0+d_2f_2
\end{equation}
for $d_0,d_2\in k$. With this, we note that $\{f_0,f_2\}$ is then a basis for $V$.

With this and (\ref{syz sum}), it follows that
  $$f_0u^2 + (d_0f_0+d_2f_2)uv + f_2v^2 =0,$$
  hence
  \begin{equation}\label{f0 f2 Koszul syzygy}
  f_0(u^2+d_0 uv) + f_2(v^2+d_2 uv)=0.
  \end{equation}
A similar argument as before shows that $u^2+d_0 uv, v^2+d_2 uv$ is a regular sequence. Noting from (\ref{f0 f2 Koszul syzygy}) that $[f_0, f_2]$ is a syzygy on this sequence, it follows that 
\[
 \left\{
    \begin{array}{ll}
      f_0  =& h(v^2+d_2 uv) \\
      f_2 =&  -h (u^2+d_0 uv)
    \end{array}
  \right.
  \]
  for some $h \in R_{a,b-2}$. Moreover, since $u^2+d_0 uv, v^2+d_2 uv$ form a regular sequence, they have no common factor. As $v^2+d_2 uv = v(v+d_2u)$ and $u^2+d_0 uv = u(u+d_0v)$, we see that $v+d_2u$ and $u+d_0v$ are not unit multiples of each other, hence $d_0d_2-1\neq 0$.
\end{itemize}

\begin{rem}\label{Other Case}
  We purposely omit the third case that $d_0\neq 0$ in (\ref{lin dep eqn}), as it is superfluous. Indeed, if $d_0\neq 0$, then at least one of $d_1,d_2$ is nonzero as well, as $f_0\neq 0$ by \Cref{f0 and f2 nonzero}. Hence this setting belongs to at least one of the cases above. We also note that instead, one could consider the two cases that $d_1\neq 0$ and $d_2\neq 0$. However, repeating as before, this latter case follows identically to the first case above by symmetry in the monomial sequence $\{u^2,uv,v^2\}$.
\end{rem}

We now address the generation of $I_U$ in the case $\dim V=2$. In order to make notation consistent, we adjust the indices of the coefficients in the discussion above.

\begin{prop}\label{Dim 2 generation}
With the assumptions of \Cref{General Setting}, if $\dim V=2$, then after possibly reindexing, we have $I_U=(hg_0,hg_1,p_2,p_3)$ for some  $h \in R_{a,b-2}$ and $g_0,g_1\in R_{0,2}$ where either
\[
\begin{array}{ccc}
\left\{
    \begin{array}{ll}
      g_0  =& uv+d_0v^2 \\
      g_1 =&  u^2+d_1v^2
    \end{array}
  \right.
  &\text{\qquad or \qquad }
  &\left\{
    \begin{array}{ll}
      g_0  =& v^2+d_0 uv \\
      g_1 =&  u^2+d_1 uv
    \end{array}
  \right.\\[3ex]
  \text{with $d_0^2+d_1 \neq 0$} & & \text{with $d_0d_1-1 \neq 0$}
\end{array}
  \]
  for some $d_0,d_1\in k$.
\end{prop}

\begin{proof}
This follows from \Cref{Generation - dim V remark} and the previous discussion, after adjusting the indices involved.
\end{proof}

\begin{rem}\label{Find d0 and d1}
 Although the case that $\dim V=2$ splits into two subcases, one may easily determine the coefficients $d_0$ and $d_1$, and which case of \Cref{Dim 2 generation} one has. Indeed, recall that $\{p_0,p_1,p_2,p_3\}$ is linearly independent, and one has the matrix equation in \Cref{Generation - dim V remark}. Thus finding the coefficients in (\ref{lin dep eqn}) corresponds to finding a basis for the kernel of the coefficient matrix $\varphi$, which is an exercise in linear algebra.
\end{rem}

With the generating set of $I_U$ established in \Cref{Dim 2 generation}, it will be much easier to describe other syzygies of $I_U$. In particular, in \Cref{dim 2 section} we will produce a set of additional syzygies which will inevitably determine the complex of \Cref{Botbol det Z}. First however, we consider the generation of $I_U$ when $\dim V=3$.

\subsection{Case 2: \boldmath{$\dim V=3$}}

With the conditions of \Cref{General Setting}, we now consider the second case of \Cref{Dimension of V} and proceed under the assumption that $\dim V=3$, i.e. $\{f_0,f_1,f_2\}$ is a basis of $V$. As $\{f_0,f_1,f_2\}$ is linearly independent, these polynomials may be taken as minimal generators of $I_U$, following \Cref{Generation - dim V remark}.

\begin{prop}\label{Dim 3 generation}
With the assumptions of \Cref{General Setting}, if $\dim V =3$, then after possibly reindexing, we have $I_U=(\alpha u,\beta v - \alpha u, -\beta u,p_3)$ for some $\alpha, \beta \in R_{a,b-1}$.
\end{prop}

\begin{proof}
This follows from (\ref{alpha-beta columns}) and \Cref{Generation - dim V remark}.
\end{proof}

With suitable generating sets of $I_U$ established under the assumptions of \Cref{General Setting}, within both cases of \Cref{Dimension of V}, we may proceed in the following sections with a deeper study of the syzygies of $I_U$. In particular, we determine a subset of syzygies that is sufficient to determine the bigraded strand $\Z_{2a-1,b-1}$ in \Cref{Botbol det Z}, and hence determine the implicit equation of $X_U$.


\section{Syzygies in the case $\dim V=2$}\label{dim 2 section}

With the conditions of \Cref{General Setting}, we proceed under the assumption that $\dim V=2$ for the duration of this section. Recall from \Cref{Dim 2 generation} that $I_U$ may be generated as $I_U=(hg_0,hg_1,p_2,p_3)$, for particular $g_0,g_1 \in R_{0,2}$, and some $h\in R_{a,b-2}$. By selecting this generating set, we note that the syzygy $Q$ of bidegree $(0,2)$ in \Cref{General Setting} is the reduced Koszul syzygy
\begin{equation}\label{quad syz dim 2}
 Q=\begin{bmatrix}
    g_1\\-g_0\\0\\0
\end{bmatrix}.   
\end{equation}

With this specific choice of generating set for $I_U$, we aim to produce a set of additional syzygies which, with $Q$, will determine the complex $\Z_{2a-1,b-1}$  in \Cref{Botbol det Z}. Before we are able to describe these syzygies, we must verify the containment of ideals $I_U \subseteq (g_0,g_1)$, which follows from the proceeding lemma.

\begin{lemma}\label{Ideal containments}
 We have the containment of ideals $(u,v)^3 \subseteq (g_0,g_1)$. 
\end{lemma}

\begin{proof}

We write each cubic monomial of $k[u,v]$ in terms of $g_0$ and $g_1$, in both cases of \cref{Dim 2 generation}.
\begin{enumerate}
    \item[(i)] In the first case where $g_0  = uv+d_0v^2$ and $g_1 =u^2+d_1v^2$ with $d_0^2+d_1 \neq 0$, we have the following equations, which are easily verified. 
\begin{equation}\label{monomial case 1}
\arraycolsep=3pt
\begin{array}{rl}
     u^3=&  \big(\frac{-d_0d_1}{d_0^2+d_1}u - \frac{d_1^2}{d_0^2+d_1}v\big) g_0 + \big(u + \frac{d_0d_1}{d_0^2+d_1}v\big)g_1,\\[1ex]
     u^2v=& \big(\frac{d_1}{d_0^2+d_1}u - \frac{d_0d_1}{d_0^2+d_1}v\big)g_0 +\big(\frac{d_0^2}{d_0^2+d_1}v\big) g_1,  \\[1ex]
     uv^2=& \big(\frac{d_0}{d_0^2+d_1}u + \frac{d_1}{d_0^2+d_1}v\big)g_0 -\big(\frac{d_0}{d_0^2+d_1}v \big) g_1, \\[1ex]
     v^3=& \big(\frac{-1}{d_0^2+d_1}u + \frac{d_0}{d_0^2+d_1}v\big)g_0 +\big(\frac{1}{d_0^2+d_1}v\big) g_1.
\end{array}
\end{equation}

    \item[(ii)] 
Similarly, if $g_0  = v^2+d_0 uv$ and $g_1 =  u^2+d_1 uv$ with $d_0d_1-1 \neq 0$, we have the following.
\begin{equation}\label{monomial case 2}
\arraycolsep=3pt
\begin{array}{rl}
     u^3=& \big(\frac{-d_1^2}{d_0d_1-1}u\big)g_0 +\big(u+\frac{d_1}{d_0d_1-1}v\big)g_1, \\[1ex]
     u^2v=&\big(\frac{d_1}{d_0d_1-1}u\big)g_0 -\big(\frac{1}{d_0d_1-1}v\big) g_1,\\[1ex]
     uv^2=& \big(\frac{-1}{d_0d_1-1}u\big) g_0 +\big(\frac{d_0}{d_0d_1-1}v\big)g_1, \\[1ex]
     v^3=&\big(\frac{d_0}{d_0d_1-1}u +v \big)g_0 -\big(\frac{d_0^2}{d_0d_1-1}v\big)g_1.
\end{array}
\end{equation}
\end{enumerate}
Hence we have $(u,v)^3 \subseteq (g_0,g_1)$, in each case of \Cref{Dim 2 generation}.\qedhere
\end{proof}

\begin{rem}
From \Cref{Ideal containments}, we have that $I_U \subseteq (g_0,g_1)$, as $I_U$ is generated in bidegree $(a,b)$ with $b\geq 3$. Hence one has $V(g_0,g_1) \subseteq V(I_U)$, and we note that this is not a contradiction to the assumption that $U$ is basepoint free. Indeed, recall that $g_0, g_1$ is a regular sequence in $k[u,v]$, hence it follows that $V(g_0,g_1)=\emptyset$ in $\P^1$. Alternatively, one may also see that $V(g_0)$ and $V(g_1)$ have no intersection from \Cref{Hartshorne lemma}.
\end{rem}

We now produce an additional pair of syzygies that, along with $Q$, determine the implicit equation of $X_U$. Following the discussion in \Cref{lin syz case}, we note that the statement and proof below are remarkably similar to those of \cite[2.2]{DS16}. We revisit this observation in \Cref{Questions section}, by making a more general conjecture.

\begin{thm}\label{Dim 2 main result thm}
With the conditions of \Cref{General Setting}, assume that $\dim V=2$. The ideal $I_U$ has two syzygies $S_1$, $S_2$ of bidegree $(a,b-2)$ such that $\dim \langle Q,S_1,S_2\rangle_{2a-1,b-1} =2ab$.
\end{thm}

\begin{proof}
We proceed in a manner similar to the proof of \cite[2.2]{DS16}. Recall from \Cref{Dim 2 generation} that we may take $p_0=hg_0$ and $p_1=hg_1$. Moreover, as $(p_0,p_1,p_2,p_3)=I_U \subseteq (g_0,g_1)$ by \Cref{Ideal containments}, we may write 
\begin{equation}\label{p2,p3 decomp}
\left\{\begin{array}{rl}
     p_2 &= q_0g_0 +q_1g_1  \\
     p_3 &= r_0g_0 +r_1g_1
\end{array}
\right.
\end{equation}
for some $q_0,q_1,r_0,r_1 \in R_{a,b-2}$. With this, note that $q_0 p_0+q_1p_1 -hp_2 =0$, and also $r_0 p_0+r_1p_1 -hp_3 =0$. Hence both
\[
S_1=\begin{bmatrix}
    q_0\\
    q_1\\
-h\\
    0
\end{bmatrix}\qquad\text{and} \qquad S_2=\begin{bmatrix}
    r_0\\
    r_1\\
0\\
    -h
\end{bmatrix}
\]
are syzygies of $I_U$. Thus the syzygy module of $I_U$ contains the span of the columns of
\begin{equation}\label{Dim 2 M syzygy matrix}
M=\begin{bmatrix}
g_1 & q_0 & r_0\\
-g_0 & q_1 & r_1\\
0 & -h & 0\\
0 & 0 &-h
\end{bmatrix}.
\end{equation}
Deleting the first row of $M$ yields an upper triangular matrix, hence the columns of $M$ span a free $R$-module. The claim will follow once it has been shown that $M_{2a-1,b-1}$ consists of $2ab$ linearly independent columns, the independence following from this previous observation.

We note that the number of columns contributed to $M_{2a-1,b-1}$ and also a matrix representation of $d_1$ in $\Z_{2a-1,b-1}$ by each syzygy agree. Moreover, this is a matter of counting monomials in certain bidegrees. Indeed, as the syzygy $Q$ has bidegree $(0,2)$, it yields
\begin{equation}\label{dim V=2, (0,2) contribution}
h^0(\O_{\P^1\times \P^1}(2a-1,b-3)) = 2a(b-2)
\end{equation}
columns of $d_1$. Similarly, as both $S_1$ and $S_2$ are syzygies of bidegree $(a,b-2)$, they each give rise to
\begin{equation}\label{dim V=2, (a,b-2) contribution}h^0(\O_{\P^1\times \P^1}(a-1,1)) = 2a
\end{equation}
columns of $d_1$. Moreover, the fact that these columns are linearly independent follows as $\{Q,S_1,S_2\}$ spans a free $R$-module.
\end{proof}

\begin{cor}\label{Dim 2 main result cor}
With the assumptions of \Cref{Dim 2 main result thm}, the first differential $d_1$ of the bigraded strand $\Z_{2a-1,b-1}$ of the approximation complex $\Z$ is determined by the syzygies $\{Q,S_1,S_2\}$.
\end{cor}

\begin{proof}
    This follows from \Cref{Dim 2 main result thm}, \Cref{Botbol det Z}, and \Cref{Botbol implicit eqn}.
\end{proof}

We conclude this section with an example, using the syzygies of $I_U$ constructed in \Cref{Dim 2 main result thm} to determine the implicit equation of $X_U$. We note that all of the necessary tools are in \Cref{Prelim Section}.

\begin{ex}\label{example continued}
We continue and finish \Cref{Intro example dim V=2} from the introduction, noting that the columns of (\ref{intro example syzygies}) are precisely the syzygies constructed in the proof of \Cref{Dim 2 main result thm}. With this, the bidegree $(2a-1,b-1) = (3,2)$ component of (\ref{intro example syzygies}) is generated by the image of 
\[
 \arraycolsep=2.5pt\def\arraystretch{0.9}
\left[ \begin{array}{cccccccccccccccccc}
s^3v^2&s^2tv^2&st^2v^2&t^3v^2&0&0&0&0&st^2u^2&st^2uv&t^3u^2&t^3uv\\
      -s^3u^2&-s^2tu^2&-st^2u^2&-t^3u^2&s^3uv&s^3v^2&s^2tuv&s^2tv^2&0&0&0&0\\
      0&0&0&0&-suh& -svh&-tuh&-tv h&0&0&0&0\\
      0&0&0&0&0&0&0&0&-su h&-svh&-tuh&-tvh
      \end{array}\right]
\]
where $h=s^2u+t^2v$. The fact that each syzygy contributes four columns follows from (\ref{dim V=2, (0,2) contribution}) and (\ref{dim V=2, (a,b-2) contribution}). Multiplying by $[T_0, T_1,T_2,T_3]$ and contracting against the monomials of $R$ in bidegree $(3,2)$ shows that a matrix representation of $d_1$ in the bigraded strand $\Z_{3,2}$ is the $12\times 12$ matrix
\[
\arraycolsep=1.5pt\def\arraystretch{0.9}
d_1 =\left[ \begin{array}{cccccccccccc}
      -T_1&0&0&0&-T_2&0&0&0&-T_3&0&0&0\\
      0&0&0&0&T_1&-T_2&0&0&0&-T_3&0&0\\
      T_0&0&0&0&0&T_1&0&0&0&0&0&0\\
      0&-T_1&0&0&0&0&-T_2&0&0&0&-T_3&0\\
      0&0&0&0&0&0&T_1&-T_2&0&0&0&-T_3\\
      0&T_0&0&0&0&0&0&T_1&0&0&0&0\\
      0&0&-T_1&0&0&0&0&0&T_0&0&0&0\\
      0&0&0&0&-T_2&0&0&0&-T_3&T_0&0&0\\
      0&0&T_0&0&0&-T_2&0&0&0&-T_3&0&0\\
      0&0&0&-T_1&0&0&0&0&0&0&T_0&0\\
      0&0&0&0&0&0&-T_2&0&0&0&-T_3&T_0\\
      0&0&0&T_0&0&0&0&-T_2&0&0&0&-T_3\end{array}\right]
\]
with determinant
$$(T_0^3T_1^3-T_0^4T_2^2-2\,T_0^2T_1^2T_2T_3-T_1^4T_3^2)^2.$$

By \Cref{Botbol det Z} and \Cref{Botbol implicit eqn}, the implicit equation of $X_U$ is $F=T_0^3T_1^3-T_0^4T_2^2-2\,T_0^2T_1^2T_2T_3-T_1^4T_3^2$ and the degree of the rational map defined by $U$ is $\deg \phi_U= 2$.
\end{ex}

\begin{rem}
Following the procedure above, one needs only to determine the syzygies $\{Q,S_1,S_2\}$ to determine the implicit equation of $X_U$. Following their construction in \Cref{Dim 2 main result thm}, recall from \Cref{Find d0 and d1} that the polynomials $g_0$, $g_1$, and $h$ of \Cref{Dim 2 generation} may be found easily from the coefficient matrix (\ref{coefficient matrix}). Additionally, one may write $p_2$ and $p_3$ in terms of the monomial basis $\{u^3,u^2v,uv^2,v^3\}$, and then use the equations of \Cref{Ideal containments} to find suitable $q_0,q_1,r_0,r_1$ in (\ref{p2,p3 decomp}). In particular, this method is significantly less computationally intensive than computing the entire syzygy module of $I_U$ to produce $\Z_{2a-1,b-1}$.
\end{rem}


\section{Syzygies in the case $\dim V=3$}\label{dim 3 section}

We now consider the second case of \Cref{Dimension of V} and proceed under the conditions of \Cref{General Setting}, with the assumption that $\dim V =3$. As before, we aim to produce a subset of syzygies of $I_U$ that determines the first differential of $\Z_{2a-1,b-1}$, and hence the implicit equation of $X_U$, following \Cref{Botbol implicit eqn}.

With the assumption that $\dim V=3$, recall from \Cref{Dim 3 generation} that the ideal $I_U$ may be generated as $I_U=(\alpha u,\beta v - \alpha u, -\beta u,p_3)$ for some $\alpha, \beta \in R_{a,b-2}$. By selecting this generating set, the syzygy $Q$ of bidegree $(0,2)$ in \Cref{General Setting} is then
\begin{equation}
Q=\begin{bmatrix}\label{Quad Syz dim V=3}
    u^2\\uv\\v^2\\0
\end{bmatrix}    
\end{equation}
following (\ref{syz sum}).

Similar to the approach of the previous section, we begin our treatment by examining the syzygies of the subideal $(\alpha v,\beta v - \alpha u, -\beta u)$ of $I_U$, generated by the basis elements of $V$ (\ref{alpha-beta columns}). First however, we provide a short lemma, which is particularly useful for tensor product surfaces free of basepoints.

\begin{lemma}[{\cite[V.1.4.3]{Hartshorne}}]\label{Hartshorne lemma}
Let $f\in R_{a,b}$ and $g\in R_{c,d}$ such that $\gcd(f,g)=1$. The curves $V(f)$ and $V(g)$ in $\P^1\times \P^1$ meet at $ad+bc$ points.
\end{lemma}

With this, we investigate the syzygy module of the subideal $J= (\alpha v,\beta v - \alpha u, -\beta u)$ of $I_U$.

\begin{prop}\label{J g2p}
The ideal $J= (\alpha v,\beta v - \alpha u, -\beta u)$ is a perfect $R$-ideal of height 2.
\end{prop}

\begin{proof}
Since $\alpha, \beta \in R_{a,b-1}$ and $b \geq 3$, we have that $\alpha,\beta\in (u,v)^2 = (u,v^2)\cap(u^2,v)$. Hence we may write
\begin{equation}\label{alpha beta decomp}
\left\{\begin{array}{rl}
     \alpha= &  q_0u+q_1v^2  \\
     \beta= &  r_0u^2 +r_1v
\end{array}
\right.
\end{equation}
for some $q_0,r_1 \in R_{a,b-2}$ and $q_1,r_0 \in R_{a,b-3}$.
With this, notice that the generators of $J$ are precisely the signed $2\times 2$ minors of 
\begin{equation}\label{HB matrix}   
\begin{bmatrix}
u^2&-q_1u +r_1 \\
uv & -r_0u -q_1v\\
v^2& q_0-r_0v
\end{bmatrix}
\end{equation}
and so the claim will follow from the Hilbert-Burch theorem \cite[20.15]{Eisenbud}, once it has been shown that $\hgt J\geq 2$. 

By \Cref{f0 and f2 nonzero} we have that $J\neq 0$, hence it suffices to show that $\hgt J \neq 1$. Suppose, to the contrary, that $\hgt J=1$. Thus $f_0,f_1,f_2$ have a non-unit common factor, and write $h$ to denote the greatest common factor. Notice that $I_U\subseteq (h,p_3)$, hence $h$ and $p_3$ have no common factor, as $ \hgt I_U =2$ since $\sqrt{I_U} = (s,t)\cap (u,v)$. However, as $h$ is a non-unit, we may write $\bideg h= (c,d)$ for either $c\geq 1$ or $d\geq 1$ and, from the containment of ideals, we have $V(I_U) \supseteq V(h,p_3)$. Thus by \Cref{Hartshorne lemma} it follows that $V(I_U)$ contains $ad+bc >0$ points, which contradicts the assumption that $I_U$ is basepoint free.
\end{proof}

\begin{cor}\label{alpha beta reg seq.}
The polynomials $\alpha, \beta$ form an $R$-regular sequence.   
\end{cor}

\begin{proof}
This follows from \Cref{J g2p}, noting that $J\subseteq (\alpha, \beta)$.
\end{proof}

We may now state the main result of this section. Similar to \Cref{Dim 2 main result thm}, we produce a set of additional syzygies which, with $Q$ in (\ref{Quad Syz dim V=3}), are enough to determine the implicit equation of $X_U$.

\begin{thm}\label{Dim 3 main result thm}
With the assumptions of \Cref{General Setting}, assume that $\dim V=3$. The ideal $I_U$ has a syzygy $S_1$ of bidegree $(a,b-2)$ and two syzygies $S_2$, $S_3$ of bidegree $(a,b-1)$ such that $\dim \langle Q,S_1,S_2,S_3\rangle_{2a-1,b-1} =2ab$.
\end{thm}

\begin{proof}
Recall from \Cref{Dim 3 generation} that we may take $p_0=\alpha v$, $p_1=\beta v - \alpha u$, and $p_2=-\beta u$. With this, we begin with the syzygy in bidegree $(a,b-2)$. By \Cref{J g2p} and the Hilbert-Burch theorem \cite[20.15]{Eisenbud}, the matrix (\ref{HB matrix}) is precisely the syzygy matrix of $J= (\alpha v,\beta v - \alpha u, -\beta u)$. Hence, we may extend its columns to syzygies on $I_U$. Doing so yields $Q$ in (\ref{Quad Syz dim V=3}) and also
\[
S_1=\begin{bmatrix}
-q_1u +r_1 \\
-r_0u -q_1v\\
q_0-r_0v\\
0
\end{bmatrix}
\]
as a syzygy of $I_U$ in bidegree $(a,b-2)$.

For the remaining two syzygies, we must involve the last generator of $I_U$. As $\bideg p_3 = (a,b)$ and $b\geq 3$, we note that $p_3\in (u^2,v^2)$. Thus we may write 
\begin{equation}\label{dim 3 - p3 decomp}
 p_3= m_0u^2+m_1v^2   
\end{equation}
for some $m_0,m_1 \in R_{a,b-2}$. With this, notice that $-m_1v p_0+m_0up_1+m_0vp_2+\alpha p_3=0$ and also $m_1u p_0+m_1vp_1-m_0up_2-\beta p_3=0$. Hence 
\[
S_2=\begin{bmatrix}
    -m_1v\\ m_0u\\ m_0 v\\ \alpha
\end{bmatrix} \qquad\text{and} \qquad S_3=\begin{bmatrix}
    m_1u\\ m_1v\\ -m_0u\\-\beta
\end{bmatrix}
\]
are syzygies of $I_U$ in bidegree $(a,b-2)$.

Consider the matrix 
\begin{equation}
M=\begin{bmatrix}
 u^2& -q_1u +r_1 &-m_1v &m_1u\\
uv &-r_0u -q_1v &m_0u &m_1v\\
v^2& q_0-r_0v &m_0v&-m_0u\\
0  &0 &\alpha & -\beta
\end{bmatrix}
\end{equation}
with columns $\{Q,S_1,S_2,S_3\}$. We first show that the columns of $M_{2a-1,b-1}$ are linearly independent, i.e. $M_{2a-1,b-1}$ is injective. We note that $M$ itself is not injective, and we claim that its kernel is spanned by
\[
N=\begin{bmatrix}
m_0(q_1u-r_1) +m_1(r_0v-q_0) \\
p_3\\
\beta\\
\alpha
\end{bmatrix}.
\]
To verify this, consider the sequence of bigraded $R$-modules
\begin{equation}\label{MN complex}
 0 \rightarrow R(-2a,-2b+2) \overset{N}{\longrightarrow} \begin{array}{c}
 R(0,-2)\\
 \oplus \\
R(-a,-b+2)\\
 \oplus \\
R(-a,-b+1)^2\\
\end{array} \overset{M}{\longrightarrow} \,
R^4.
\end{equation}
A direct computation shows this is a complex. Moreover, this complex is exact by the Buchsbaum-Eisenbud acyclicity criterion \cite[Cor. 1]{BE73}, with the required conditions following from \Cref{J g2p} and \Cref{alpha beta reg seq.}. As (\ref{MN complex}) is a bigraded free complex, we may consider the bigraded strand in bidegree $(2a-1,b-1)$. By degree considerations it then follows that $N_{2a-1,b-1} =0$, hence $M_{2a-1,b-1}$ is indeed injective.

 Now that the columns of $M_{2a-1,b-1}$ have been shown to be linearly independent, we need only count them to verify the assertion. As previously noted, this is also the number of columns contributed to a matrix representation of $d_1$, by the syzygies above. The syzygy $Q$ has bidegree $(0,2)$, hence it yields
$$h^0(\O_{\P^1\times \P^1}(2a-1,b-3)) = 2a(b-2)$$
columns of $d_1$. Similarly, the syzygy $S_1$ in bidegree $(a,b-2)$ yields
$$h^0(\O_{\P^1\times \P^1}(a-1,1)) = 2a$$
columns. Lastly, the syzygies $S_2$ and $S_3$ in bidegree $(a,b-1)$ each give rise to
$$h^0(\O_{\P^1\times \P^1}(a-1,0)) = a$$
columns of $d_1$, and the claim follows.
\end{proof}

\begin{cor}\label{Dim 3 main result cor}
With the assumptions of \Cref{Dim 3 main result thm}, the first differential $d_1$ of the bigraded strand $\Z_{2a-1,b-1}$ of the approximation complex $\Z$ is determined by the syzygies $\{Q,S_1,S_2,S_3\}$. 
\end{cor}

\begin{proof}
    This follows from \Cref{Dim 3 main result thm}, \Cref{Botbol det Z}, and \Cref{Botbol implicit eqn}.
\end{proof}

\begin{rem}
Once the syzygies $\{Q,S_1,S_2,S_3\}$ have been constructed, the procedure to determine a matrix representation of $d_1$ in $\Z_{2a-1,b-1}$, and obtain the implicit equation of $X_U$, is the same as the process in \Cref{example continued}. Following \Cref{Dim 3 main result thm}, we note that the entries of these syzygies can be found easily from the decompositions (\ref{alpha-beta columns}), (\ref{alpha beta decomp}), and (\ref{dim 3 - p3 decomp}). As a consequence, finding the implicit equation with this procedure is computationally simple, compared to computing the full syzygy module of $I_U$ in order to produce $\Z_{2a-1,b-1}$.
\end{rem}


\section{Further observations and questions}\label{Questions section}

In this final section, we address some additional questions and possibilities for future directions, related to the results presented here. As the primary technique of this article is the construction of the subspace $V$ associated to a given syzygy, one natural question is how to extend this method to more general settings. A further question is whether $V$, or rather its dimension, is always a sufficient invariant, as it was here.

The first question is easily answered. Suppose one has $U\subseteq H^0(\O_{\P^1\times \P^1}(a,b))$ with basis $\{p_0,p_1,p_2,p_3\}$, as before. If $I_U$ has a syzygy $C$ with entries in $R=k[s,t,u,v]$ of a given bidegree, say $(c,d)$, then one may repeat the process described in \Cref{Quad Syz Section}. Letting $\{m_0,\ldots,m_n\}$ denote the monomials in $R_{c,d}$, one has
$$\sum_{j=0}^3 \big(\sum_{i=0}^n a_{ij} m_i\big) p_j=0$$
for some $a_{ij}\in k$. This may be rearranged as
$$0=\sum_{j=0}^3 \big(\sum_{i=0}^n a_{ij} m_i\big) p_j = \sum_{i=0}^n \big( \sum_{j=0}^3 a_{ij} p_j\big) m_i= \sum_{i=0}^n f_i m_i,$$
by letting $f_i=\sum_{j=0}^3 a_{ij} p_j$, for $0\leq i\leq n$. Then it is clear that $[f_0,\ldots,f_n]$ is a syzygy of the monomial ideal $J=(m_0,\ldots,m_n)$. Hence one proceeds by examining $\syz(J)$, noting that a free resolution (non-minimal in general) of $J$ is readily available from Taylor's resolution \cite[Ex. 17.11]{Eisenbud}. In particular, the study of $\syz(J)$ is very approachable in this setting.

Moreover, one may write $V$ to denote the subspace $V=\Span_k\{f_0,\ldots,f_n\}$ of $U$. An argument similar to the proof of \Cref{Dimension of V} shows that $2\leq \dim V\leq 4$. Moreover, similar to \Cref{Generation - dim V remark}, $\dim V$ agrees with the rank of the coefficient matrix $A=(a_{ij})$, and this is the maximum number of the $f_i$ which may be taken as part of a minimal generating set of $I_U$.

Whereas the formulation of such a subspace $V$ associated to a syzygy is quite general, we note that its dimension is rarely a sufficient invariant, as it was in \Cref{dim 2 section,dim 3 section}, even for syzygies of low total degree. We illustrate this next in a setting where this fails to be the case. However, we then consider a setting where it is believed that $\dim V$ is the correct invariant, in an effort to extend the work presented here and in \cite{DS16}.

\subsection{Syzygies of bidegree \boldmath{$(1,1)$}}
Recall that the assumptions of \Cref{General Setting} may be modified to address tensor product surfaces with a syzygy in bidegree $(2,0)$. Thus the natural question, which has not yet been answered, is how to proceed in the remaining case of a quadratic syzygy, namely in bidegree $(1,1)$. We may proceed in the manner above, however we will see that the behavior of the remaining syzygies is more erratic.

If the ideal $I_U \subseteq R=k[s,t,u,v]$ has such a syzygy in bidegree $(1,1)$, we may write
$$\sum_{i=0}^3(a_isu+b_isv+c_itu+d_itv) p_i=0.$$
for some $a_i,b_i,c_i,d_i \in k$. Rewriting this, we have
\begin{equation}\label{11 syzygy sum}
0=\sum_{i=0}^3(a_isu+b_isv+c_itu+d_itv)p_i= (\sum_{i=0}^3a_i p_i)su + (\sum_{i=0}^3 b_i p_i)sv + ( \sum_{i=0}^3 c_i p_i) tu +  ( \sum_{i=0}^3 d_i p_i) tv.
\end{equation}
Writing $f_0 = \sum_{i=0}^3a_i p_i$, $f_1=\sum_{i=0}^3b_i p_i$, $f_2 = \sum_{i=0}^3c_i p_i$, and $f_3 = \sum_{i=0}^3d_i p_i$, we see that $[f_0, f_1, f_2,f_3]$ is a syzygy on $(su,sv,tu,tv)$. A resolution of this ideal is easily computed, from which one has
\[
\begin{bmatrix}
    f_0\\
    f_1\\
    f_2\\
    f_3
\end{bmatrix} = \alpha \begin{bmatrix}
    v\\
    -u\\
    0\\
    0
\end{bmatrix} + \beta \begin{bmatrix}
    0\\
    0\\
    v\\
    -u
\end{bmatrix}
+ \gamma\begin{bmatrix}
    t\\
    0\\
    -s\\
    0
\end{bmatrix}
+\delta\begin{bmatrix}
    0\\
    t\\
    0\\
    -s
\end{bmatrix}
=\begin{bmatrix}
    \alpha v +\gamma t\\
    -\alpha u +\delta t\\
    \beta v-\gamma s\\
    -\beta u -\delta s
\end{bmatrix}
\]
for $\alpha,\beta \in R_{a,b-1}$ and $\gamma,\delta \in R_{a-1,b}$. 

Writing $V = \Span_k\{f_0,f_1,f_2,f_3\}$, we show that $\dim V$ alone is inadequate to dictate the behavior of the remaining syzygies of $I_U$ in this setting. We consider the following two examples, with $\dim V=4$ in both, and so $I_U$ may be generated as $I_U=(f_0,f_1,f_2,f_3)$.

\begin{ex}\label{11-Syz Example 1}
Consider $U \subseteq H^0(\O_{\P^1\times \P^1}(3,3))$ with basis $\{p_0,p_1,p_2,p_3\}$ where
\[
\begin{array}{rl}
     p_0=& -2t^3u^2v-s^2tv^3-t^3v^3 \\
     p_1=& -s^2tu^3+t^3u^3\\
     p_2=& -s^2tuv^2+s^3v^3+st^2v^3-t^3v^3 \\
     p_3=& s^3u^3+st^2u^3+s^2tu^2v+t^3uv^2.
\end{array}
\]
A computation through \texttt{Macaulay2} \cite{Macaulay2} shows that $U$ is basepoint free and $I_U$ has a syzygy in bidegree $(1,1)$. Formulating the subspace $V$ as above, one computes the rank of the coefficient matrix and sees that $\dim V=4$. The ideal $I_U$ has 10 minimal first syzygies in bidegree 
\[
\begin{array}{cccccccccc}
(1,1),&(0,5),&(2,3),&(2,3),&(3,2),&(3,2),&(5,2),&(5,2),&(6,1),&(6,1).
\end{array}
\]
In particular, the only syzygies which can contribute to a matrix representation of $d_1$ in $\Z_{5,2}$ are the five in bidegrees $(1,1)$, $(3,2)$, and $(5,2)$. Further computations show that this is the case, and these five syzygies do determine $d_1$, and hence the implicit equation, in a manner similar to the process outlined in \Cref{example continued}.
\end{ex}

\begin{ex}\label{11-Syz Example 2}
Consider $U \subseteq H^0(\O_{\P^1\times \P^1}(3,3))$ with basis $\{p_0,p_1,p_2,p_3\}$ where
\[
\begin{array}{rl}
     p_0=& -t^3u^3-s^3u^2v-st^2u^2v-st^2uv^2 \\
     p_1=& s^3u^3+st^2u^3-st^2uv^2-t^3v^3\\
     p_2=& st^2u^3-s^2tu^2v+s^2tuv^2-s^3v^3\\
     p_3=& s^2tu^3+s^3uv^2+s^2tuv^2+st^2v^3.
\end{array}
\]
Similar computations show that $U$ is basepoint free, $I_U$ has a syzygy in bidegree $(1,1)$, and also $\dim V=4$. The ideal $I_U$ has 10 minimal first syzygies in bidegree    
\[
\begin{array}{cccccccccc}
(1,1),& (1,9), &(1,5),& (2,5), & (2,3),& (2,3),& (3,2),& (3,2),&  (5,1),& (7,1).
\end{array}
\]
Moreover, the only syzygies which can contribute to $d_1$ in $\Z_{5,2}$ are the four in bidegrees $(1,1)$, $(3,2)$, and $(5,1)$. As before, a further computation shows that these syzygies do determine $d_1$ in this example.
\end{ex}

\begin{rem}
    In each of the previous examples, we see that $U\subseteq H^0(\O_{\P^1\times \P^1}(3,3))$ is basepoint free, $I_U$ has a syzygy in bidegree $(1,1)$, and $\dim V=4$. However, in the examples above, the differential $d_1$ is derived from a different number of syzygies in differing bidegrees. Hence one cannot achieve a result analogous to \Cref{Dim 2 main result thm} or \Cref{Dim 3 main result thm} in this setting, based on $\dim V$ alone.
\end{rem}

\subsection{Syzygies of bidegree \boldmath{$(0,n)$}}

From the behavior in \Cref{11-Syz Example 1,11-Syz Example 2}, it is apparent that the results presented in \Cref{dim 2 section,dim 3 section} are not a consequence of the low total degree of the initial syzygy $Q$ in bidegree $(0,2)$. Rather, their success is likely due to the fact that $Q$ consists of homogeneous entries in the subring $k[u,v]$. In this setting, the ideal of entries of $Q$ is much simpler, as are its syzygies. With this, we briefly discuss the case that $I_U$ has a syzygy of bidegree $(0,n)$. 

Similar to \Cref{General Setting}, we may assume that $n$ is minimal, i.e. $I_U$ has no syzygy in bidegree $(0,m)$ for $m<n$. Moreover, we also assume that $U\subseteq R_{a,b}$ is basepoint free with $b\geq n+1$. Following the approach of \Cref{Quad Syz Section}, we have
$$\sum_{j=0}^3 \big(\sum_{i=0}^n a_{ij} u^{n-i}v^i\big) p_j =0$$
for some $a_{ij}\in k$. Once more, we rearrange this as
$$0=\sum_{j=0}^3 \big(\sum_{i=0}^n a_{ij} u^{n-i}v^i\big) p_j = \sum_{i=0}^n \big( \sum_{j=0}^3 a_{ij} p_j\big) u^{n-i}v^i= \sum_{i=0}^n f_i u^{n-i}v^i$$
where $f_i = \sum_{j=0}^3 a_{ij} p_j$. Thus $[f_0,f_1,\ldots,f_n]$ is a syzygy of the ideal $J=(u^n,u^{n-1}v,\ldots,v^n)$. As a resolution of this ideal is easily produced for any $n$, say by the Hilbert-Burch theorem \cite[20.15]{Eisenbud}, one notes that $\syz(J)$ is free and spanned by the columns of the $(n+1)\times n$ matrix
\[
\begin{bmatrix}
    v& \\
    -u & v\\
    & -u & \scalebox{0.8}{$\ddots$}\\
    & &  \scalebox{0.8}{$\ddots$} & v\\
 & & & -u\\
\end{bmatrix},
\quad \text{and so} \quad
\begin{bmatrix}
    f_0\\
    f_1\\
    \vdots\\
    f_{n-1}\\
    f_n
\end{bmatrix} = 
\begin{bmatrix}
    \alpha_1 v\\
    \alpha_2v -\alpha_1 u \\
    \vdots\\
    \alpha_nv-\alpha_{n-1}u\\
    -\alpha_n u
\end{bmatrix}
\]
for some $\alpha_1,\ldots\alpha_n \in R_{a,b-1}$. Write $V= \Span_k\{f_0,\ldots,f_n\}$ and note that $2\leq\dim V \leq 4$ following the argument of \Cref{Dimension of V}. Moreover, notice that $\dim V$ agrees with the rank of the coefficient matrix $A=(a_{ij})$, as in \Cref{Generation - dim V remark}.

It is suspected, with support from experimentation through \texttt{Macaulay2} \cite{Macaulay2}, that $\dim V$ is the correct invariant to distinguish between the behavior of the syzygies required to determine $\Z_{2a-1,b-1}$ in this setting. If correct, there are three cases to consider, based on the dimension of $V$. However, each case is likely fraught with subcases, similar to \Cref{dim 2 section}, perhaps making the approach presented here impractical. Nevertheless, the author intends to study this setting in a future paper.

Before we conclude this article, we present a conjecture within this setting in the case that $\dim V=2$, based on observations made here.

\begin{conj}
Suppose that $U\subseteq R_{a,b}$ is basepoint free with $b\geq n+1$ and $I_U$ has a syzygy $C$ in bidegree $(0,n)$, and no syzygy in $k[u,v]$ of smaller degree. If $\dim V=2$, then $I_U$ has two syzygies $S_1, S_2$ of bidegree $(a,b-n)$ such that the first differential $d_1$ of $\Z_{2a-1,b-1}$, and hence the implicit equation of $X_U$, is determined by $\{C,S_1,S_2\}$. 
\end{conj}

We note that this conjecture is true in the case $n=1$ by \cite[2.2]{DS16} following \Cref{lin syz case}, as well as the case that $n=2$ by \Cref{Dim 2 main result thm}. If confirmed, it is curious if there are similar results for the cases $\dim V=3$ and $\dim V=4$, when $I_U$ has a syzygy from the subring $k[u,v]$.


\section*{Acknowledgements}

The use of \texttt{Macaulay2} \cite{Macaulay2} was helpful in the preparation of this article, offering numerous examples to support the results presented here, and verifying the direct computations and equations in \Cref{dim 2 section,dim 3 section}.


\end{document}